\documentclass[sn-mathphys-num]{sn-jnl}
\usepackage[caption = false]{subfig}
\usepackage{inconsolata}
\usepackage{graphicx}
\usepackage[T1]{fontenc}
\usepackage[noend]{algpseudocode}
\usepackage{algorithmicx,algorithm}
\usepackage{enumitem}
\usepackage{xcolor}
\usepackage[title]{appendix}
\usepackage{multirow}
\usepackage{amsmath,amssymb,amsfonts}
\usepackage{amsthm}
\usepackage{textcomp}
\usepackage{manyfoot}
\usepackage{booktabs}
\usepackage{tabularx}
\usepackage{algpseudocode}
\usepackage{listings}
\DeclareMathOperator{\tr}{tr}

\newtheorem{remark}{Remark}[section]

\newtheorem{theorem}{Theorem}[section]
\newtheorem{corollary}{Corollary}[section]
\newtheorem{proposition}{Proposition}[section]

\usepackage{mathrsfs}

\usepackage{enumitem}

\begin{document}
\title[An Analysis on Stochastic Lanczos Quadrature with Asymmetric Quadrature Nodes]{An Analysis on Stochastic Lanczos Quadrature with Asymmetric Quadrature Nodes}
% \thanks{Submitted to the editors DATE.
% \funding{This work was funded by the natural science foundation of China (12271047); Guangdong Provincial Key Laboratory of Interdisciplinary Research and Application for Data Science, BNU-HKBU United International College (2022B1212010006); UIC research grant (R0400001-22; UICR0400008-21; R72021114); Guangdong College Enhancement and Innovation Program (2021ZDZX1046).}}

%\usepackage[A5paper]{geometry}
%\usepackage[a4paper, total={6in, 8in}]{geometry}
% Authors: full names plus addresses.
\author{Wenhao Li}
%\thanks{Guangdong Provincial Key Laboratory of Interdisciplinary Research and Application for Data Science;
%Department of Applied Mathematics, BNU-HKBU United International College, Zhuhai 519087, P.R. China.}

%\thanks{Research Center for Mathematics, Beijing Normal University, Zhuhai 519087, P.R. China.}
\author{Yixuan Huang} 
\author{Shengxin Zhu}

\abstract{This paper revisits the error analysis of the Stochastic Lanczos Quadrature (SLQ) method for approximating the trace of matrix functions, with a specific focus on asymmetric Lanczos quadrature rules. We reexplain an existing theoretical discrepancy regarding the necessity of a scaling factor when applying an affine transformation from the reference interval to the physical spectral interval. Furthermore, we introduce an optimized error reallocation technique for log-determinant estimation. Rather than evenly splitting the error tolerance between the Hutchinson trace estimator and the Lanczos quadrature, we formulate an optimization problem to strategically distribute the error budget. This approach minimizes the total number of matrix-vector multiplications (MVMs) required to reach a target accuracy for both Rademacher and Gaussian queries. Numerical experiments validate that this reallocation yields tighter theoretical bounds and provides a concrete rule-of-thumb for parameter configuration: to achieve a target accuracy efficiently, more computational resources should be allocated to the Lanczos process (larger $m$) rather than Monte Carlo sampling (smaller $N$).}



\maketitle

\section{Background}
\label{sec:background}
The Stochastic Lanczos Quadrature (SLQ) method, which combines the stochastic trace estimator \cite{H90} with Lanczos quadrature \cite{GM09}, has become a prominent approach for approximating the trace of matrix functions. It was first introduced by \cite{BFG96,BG96}, and its first rigorous theoretical analysis was provided by Ubaru, Chen, and Saad \cite{UCS17}. Building on that foundation, Cortinovis and Kressner further investigated the SLQ algorithm from a theoretical perspective \cite{CK21}. Both works follow the framework of polynomial approximation error analysis for analytic functions via Chebyshev expansion, but they differ in two main aspects, as well as in notation.

First, \cite{CK21} pointed out that in the Lanczos framework, the integration error for odd-degree Chebyshev polynomials is not always zero. This enlarges the error bound given in \cite{UCS17} when the Gauss quadrature rule is asymmetric under the Lanczos process. Another work of ours studies the necessary and sufficient condition for symmetric Lanczos quadrature \cite{LZ25}.

Second, \cite{CK21} identified a redundant factor in the Lanczos quadrature error bound presented in \cite{UCS17}, but did not provide a detailed explanation for its presence. Note that a Lanczos quadrature rule relies on a \textbf{discontinuous} measure defined by the spectrum of the underlying matrix (see \eqref{eq:measure_f} below). Unlike the \textbf{continuous}-measure case, when an affine transformation from the reference interval $\left[-1, 1\right]$ to the physical interval $\left[\lambda_{\min}, \lambda_{\max}\right]$ is applied, the quadrature error for such a discontinuous measure remains unchanged. 

In light of these points, this paper mainly focuses on asymmetric Lanczos quadrature rules, revisiting the error analysis and the combined error bounds for log determinant estimation given by \cite{CK21}. Such computations arise in many growing applications, such as Gaussian process kernel learning \cite{RW06}, Bayesian interpolation \cite{M92}, Kullback–Leibler divergence \cite{KL51}, and linear mixed models \cite{CZNZ20, ZW18, ZW19}. 

Note that the total number of matrix‑vector multiplications (MVMs) of SLQ is the product of the number of stochastic queries $N$ and the number of Lanczos iterations $m$. Previous works have overlooked the need to properly allocate the error contributions from the stochastic trace estimator and the Lanczos quadrature. They simply split the error evenly between the two parts. In this work, we employ an optimized error allocation technique to reduce the total MVMs, providing clear guidance for selecting a suitable parameter set $\{m, N\}$. Numerical experiments confirm the benefit of reallocating computational effort between the Lanczos process and the Hutchinson trace estimator. We believe such an idea can be realized in other frameworks, e.g., Hutch++ \cite{M21}.

To explicitly highlight the practical value of our work, the main contributions are summarized as follows:
\begin{itemize}
\item We resolve a theoretical discrepancy regarding the necessity of a scaling factor when applying an affine transformation in the error analysis of asymmetric Lanczos quadrature.
\item We propose an optimized error reallocation technique that abandons the equal error split between the stochastic trace estimator and the Lanczos quadrature.
\item We provide a heuristic for selecting the parameter configuration $\{m, N\}$, significantly reducing the total number of matrix-vector multiplications (MVMs) required to achieve a target tolerance.
\end{itemize}

The paper is designed as follows. In Section \ref{sec:basics}, we review the basics of SLQ on estimating trace. Subsequently, we review and re-define the notations used in the two papers—as their usage is inconsistent—to avoid potential confusion for readers. In Section \ref{sec:ea}, we will also theoretically explain why the multiplier in the error bound is considered redundant by \cite{CK21} and review the error bound provided by \cite{CK21} for the log determinant. Based on this, we introduce an error reallocation technique, which theoretically reduces the total number of matrix‑vector multiplications required by SLQ to achieve a given accuracy, offering practical guidance. Experiments will be presented in Section \ref{sec:experiments}. We conclude in Section \ref{sec:conclusion}.

\section{Basics of the Stochastic Lanczos quadrature method}
\label{sec:basics}
For a real symmetric matrix $A \in \mathbb{R}^{n\times n}$ and a matrix function $f$ defined on the spectrum of $A$, there are various approaches to estimate the trace of matrix functions $\tr(f(A))$. Readers are referred to inducing points methods \cite{QR05,QRW07}, randomized low-rank approximation methods \cite{AAI17,LZ20}, Chebyshev polynomial-based methods \cite{HMS15, DE17}, and methods that utilize variance reduction in trace estimation, such as Hutch++ and improved algorithms based on that \cite{M21,P22,E23}. This note focuses on the stochastic Lanczos quadrature (SLQ) method \cite{BFG96,BG96,UCS17,CK21}. 

The SLQ method combines the Hutchinson trace estimator \cite{H90} and the Lanczos quadrature \cite{GM09}, where $\tr(f(A))$ can be approximated by an $N$-query Hutchinson trace estimator
\begin{equation}
    \label{eq:Hutchinson}
    \tr(f(A)) \approx \tr_{N}(f(A)) = \frac{1}{N}\sum_{i=1}^N {\boldsymbol{z}^{(i)}}^T f(A) \boldsymbol{z}^{(i)} = \frac{n}{N} \sum_{i=1}^{N} {\boldsymbol{v}^{(i)}}^T f(A) \boldsymbol{v}^{(i)},
\end{equation}
where $\boldsymbol{z}^{(i)}$ is the $i^{th}$ query vector that follows the \textit{Rademacher distribution} (i.e., every entry of the vector would either be $+1$ or $-1$ with probability 50\%) and $\boldsymbol{v}^{(i)} = \boldsymbol{z}^{(i)}/\sqrt{n}$ is the corresponding unit vector. One may also substitute $\boldsymbol{z}$ by standard Gaussian random vectors \cite{AT11}. Based on the diagonalizable $A$ and the eigen-decomposition $A=Q \Lambda Q^T$, $f(A)$ can be calculated as $f(A)=Q f(\Lambda)Q^T$. Note that $f(\Lambda)$ is a diagonal matrix with entries $f(\lambda_j), j = 1,\ldots, n$. Let $\boldsymbol{\mu} = Q^T\boldsymbol{v}$, then equation \eqref{eq:Hutchinson} further reads
\begin{equation}
    \label{eq:2.3}
    \frac{n}{N} \sum_{i=1}^{N} \ {\boldsymbol{\mu}^{(i)}}^T  f(\Lambda) \boldsymbol{\mu}^{(i)}.
\end{equation}
The term of quadratic forms in \eqref{eq:2.3} can be reformulated as a sum of Riemann Stieltjes integrals $\mathcal{I} = \int_{a}^{b} f(t) d\mu(t)$ \cite{H38},
\begin{equation}
    \begin{aligned}
    \label{eq:RS}
    \frac{n}{N} \sum_{i=1}^N {\boldsymbol{\mu}^{(i)}}^T f(\Lambda) \boldsymbol{\mu}^{(i)} & = \frac{n}{N} \sum_{i = 1}^{N} \sum_{j=1}^n f(\lambda_j)\left[{\mu}^{(i)}_{j}\right]^2 \\
    & = \frac{n}{N} \sum_{i=1}^N \int_a^b f(t) {\rm d}\mu_i(t)  = \frac{n}{N} \sum_{i = 1}^N \mathcal{I}^{(i)},
\end{aligned}
\end{equation}
where ${\mu}^{(i)}_j$ is the $j^{th}$ element of $\boldsymbol{\mu}^{(i)}$ and $\mu_i(t)$ is the corresponding piecewise measure function of 
$\mathcal{I}^{(i)}$
\begin{equation}
    \label{eq:measure_f}
        \mu_i(t)=\left\{
    \begin{aligned}
    &\quad 0 , & {\rm if} \ t < \lambda_1=a, \\
    &\sum_{j=1}^{k-1} \left[\mu^{(i)}_j\right]^2 , & {\rm if}\ \lambda_{k-1} \le t < \lambda_k, k=2,...,n, \\
    &\sum_{j=1}^{n} \left[\mu^{(i)}_j\right]^2 , & {\rm if} \ t \ge \lambda_n=b.
    \end{aligned}
    \right.
\end{equation}
According to the Gauss quadrature rule \cite[Chapter 6.2]{GM09}, a Riemann Stieltjes integral $\mathcal{I}$ can be approximated by an $m$-node Lanczos quadrature rule $\mathcal{I}_{m}$ so the last term in equation \eqref{eq:RS} reads
\begin{equation}
\label{E10}
\frac{n}{N} \sum_{i=1}^N \mathcal{I}^{(i)} \approx \frac{n}{N} \sum_{i=1}^{N} \mathcal{I}^{(i)}_m = \frac{n}{N} \sum_{i=1}^{N} \sum_{k=1}^{m} \tau^{(i)}_{k} f(\theta^{(i)}_{k}),
\end{equation}
where $m$ is the Lanczos step, and $\{ \tau^{(i)}_{k} \}_{k=1}^{m}$, $\{ \theta^{(i)}_{k} \}_{k=1}^{m}$ are the quadrature weights and nodes with respect to $\mathcal{I}^{(i)}_m$. According to \cite{GW69}, the weights $\{ \tau^{(i)}_{k} \}_{k=1}^{m}$ are computed by the squares of the $1^{st}$ elements of the normalized eigenvectors $\{\boldsymbol{y}^{(i)}_{k}\}_{k=1}^{m}$ of $T_{m}^{(i)}$, and the nodes (also known as the Ritz values in the Lanczos tri-diagonalization) $\{ \theta_k^{(i)} \}_{k=1}^{m}$ are the corresponding eigenvalues $\{\lambda_k\}_{k=1}^{m}$  of $T_{m}^{(i)}$. Algorithm \ref{alg:Lanc} \cite[Section 4.1]{GM09} outlines the process of obtaining $m$ Lanczos quadrature nodes and weights given a real symmetric matrix and an initial vector.
\begin{algorithm}[htbp]
        \raggedright
	\caption{Lanczos Algorithm for Lanczos Quadrature \cite[Section 4.1]{GM09}} 
	\label{alg:Lanc}
	\hspace*{0.02in} {\bf Input:} Real symmetric matrix $A \in \mathbb{R}^{n\times n}$, real vector $\boldsymbol{v} \in \mathbb{R}^{n}$, number of Lanczos iterations $m$.\\
	\hspace*{0.02in} {\bf Output:}
	Quadrature nodes $\{ \theta_k \}_{k=1}^{m}$ and quadrature weights $\{ \tau_k \}_{k=1}^{m}$.
	\begin{algorithmic}[1]
	    \State Normalize $\boldsymbol{v}_{1} = \boldsymbol{v}/\| \boldsymbol{v}\|_2$ and compute $\alpha_1 = {\boldsymbol{v}_{1}}^T A \boldsymbol{v}_{1}$.
	    \State Obtain $\tilde{\boldsymbol{v}}_{2} = A\boldsymbol{v}_{1} - \alpha_1 \boldsymbol{v}_{1}$.
		\State \textbf{for} $k = 2$ to $m$ 
			\State Let $\beta_{k-1} = \| \tilde{\boldsymbol{v}}^{(k)} \|_2$.
			\State Obtain $\boldsymbol{v}_k = \tilde{\boldsymbol{v}}_{k}/\beta_{k-1}$.
			\State Compute $\alpha_{k} = {\boldsymbol{v}_{k}}^T A \boldsymbol{v}_{k}$.
			\State Obtain $\tilde{\boldsymbol{v}}_{k+1} = A \boldsymbol{v}_{k} - \alpha_{k} \boldsymbol{v}_{k} - \beta_{k-1}\boldsymbol{v}_{k-1}$.
            \State \textbf{end for}
            \State Obtain $T_{m} = \left[\begin{array}{ccccc}
               \alpha_1 &  \beta_1 & 0 & \cdots & 0 \\
               \beta_1 & \alpha_2 & \beta_2 & \ddots & \vdots \\
               0 & \beta_2 & \ddots & \ddots & 0 \\
               \vdots & \ddots & \ddots & \ddots & \beta_{m-1} \\
               0 & \cdots & 0 & \beta_{m-1} & \alpha_{m}
            \end{array}\right].$
            \State Calculate the eigenpairs of $T_{m}$: $[V,D] = \texttt{eig}(T_{m})$.
		\State Obtain quadrature nodes $\theta_k = D_{kk}$, $k = 1, \ldots, m$. 
            \State Obtain quadrature weights $\tau_k = \left[ \boldsymbol{e}_1^T \boldsymbol{y}_k \right]^2$, where $\boldsymbol{e}^T_1 = (1, 0, \ldots, 0)$ and $\boldsymbol{y}_k$ is the $k^{th}$ column of $V$, $k = 1, \ldots, m$.
            \State \textbf{Return} $\{\theta_k\}_{k=1}^{m}$ and $\{\tau_k\}_{k=1}^{m}$.
	\end{algorithmic}
\end{algorithm}

Algorithm \ref{alg:SLQ} \cite[Section 3]{UCS17} summarizes how the SLQ method approximates $\tr(f(A))$.
\begin{algorithm}[htbp]
        \raggedright
	\caption{Stochastic Lanczos Quadrature Method for Trace Estimation \cite{BFG96,UCS17}}
	\label{alg:SLQ}
	\hspace*{0.02in} {\bf Input:} 
	Real symmetric matrix $A \in \mathbb{R}^{n\times n}$, number of Lanczos iterations $m$, matrix function $f$ defined on the spectrum of $A$, and number of random vectors $N$. \\
	\hspace*{0.02in} {\bf Output:} Approximation of $\tr(f(A))$.
	\begin{algorithmic}[1]
		\State \textbf{for} $i = 1$ to $N$
			\State Randomly generate a standard Gaussian/Rademacher vector $\boldsymbol{z}^{(i)} \in \mathbb{R}^{n}$ and obtain $\boldsymbol{v}^{(i)} = \frac{\boldsymbol{z}^{(i)}}{\sqrt{n}}. $
			\State Apply \ref{alg:Lanc} to $A$ with the initial vector $\boldsymbol{v}^{(i)}$ and obtain $\{\theta^{(i)}_k\}_{k=1}^{m}$ and $\{\tau_k^{(i)}\}_{k=1}^{m}$.
			\State $\Gamma \gets \Gamma + \sum_{k=1}^{m} \tau_k^{(i)}   f(\theta_k^{(i)})$.
		\State \textbf{end for}
		\State \textbf{Return} $\tr(f(A))^{\dagger} = \frac{n}{N} \Gamma$.
	\end{algorithmic}
\end{algorithm}

The error of $\tr(f(A))^{\dagger}$ consists of 
\begin{enumerate}
    \item the error brought by the Hutchinson trace estimator, \begin{equation}
        \label{err1}
        |\tr(f(A)) - \tr_N(f(A))|,
    \end{equation}
    \item the error of approximating quadratic forms by the Lanczos quadrature, 
    \begin{equation}
        \label{err2}
        |\tr_N(f(A)) - \tr(f(A))^{\dagger}|=\frac{n}{N}\sum_{i=1}^{N}|\mathcal{I}^{(i)}-\mathcal{I}_m^{(i)}|
    \end{equation}
\end{enumerate}

Avron and Toledo have derived analyses of the Gaussian estimator and Hutchinson's estimator \cite{AT11} on \eqref{err1}, while \cite{UCS17} and \cite{CK21} studied the bound of \eqref{err2}. To avoid misunderstanding, we unify the notation used by the two papers. First, \cite{UCS17} defined an $(m+1)$-node Lanczos quadrature by $\mathcal{I}_m$, while the same notation in \cite{CK21} and this work indicates a Gauss quadrature rule with a $m$-step Lanczos iteration. Second, they used different notation for analytic functions in analyzing \eqref{err2}, and $\boldsymbol{v}$ in quadratic form $\mathcal{I} = \boldsymbol{v}^T f(A)\boldsymbol{v}$ differs in normalization. We follow \cite{UCS17} that $f$ is analytic inside $\left[\lambda_{\min},\lambda_{\max}\right]$, $g$ is defined in $\left[-1,1\right]$ and $\boldsymbol{v}$ is a unit vector.

After revision, Theorem \ref{thm:UCSthm4.2} and Theorem \ref{thm:CK21} give the error bound of $|\mathcal{I}-\mathcal{I}_m|$.
\begin{theorem}
    \label{thm:UCSthm4.2} \cite[Theorem 4.2]{UCS17}
    Let $g$ be analytic in $\left[-1, 1\right]$ and analytically continuable in the open Bernstein ellipse $E_{\rho}$ with foci $\pm1$ and elliptical radius $\rho > 1$. Let $M_\rho$ be the maximum of $|g(t)|$ on $E_\rho$. Then the $m$-point Lanczos quadrature approximation satisfies
    \begin{equation}
        \label{eq:UCSthm4.2}
        |\mathcal{I}-\mathcal{I}_m| \le \frac{4M_\rho}{1 - \rho^{-2}} \rho^{-2m}.
    \end{equation}
\end{theorem}
\begin{theorem}
    \label{thm:CK21thm3} \cite[Corollary 3]{CK21}
    Let $g$ be analytic in $\left[-1, 1\right]$ and analytically continuable in the open Bernstein ellipse $E_{\rho}$ with foci $\pm1$ and elliptical radius $\rho > 1$. Let $M_\rho$ be the maximum of $|g(t)|$ on $E_\rho$. Then the $m$-point Lanczos quadrature approximation with asymmetric quadrature nodes satisfies
    \begin{equation}
        \label{eq:CK21thm3}
        |\mathcal{I} - \mathcal{I}_m| \le \frac{4M_\rho}{1-\rho^{-1}} \rho^{-2m}.
    \end{equation}
\end{theorem}
The bounds differ in the denominator. In fact, the analysis derived by \cite{UCS17} is valid for a symmetric Lanczos quadrature, whereas \cite{CK21} pointed out that Chebyshev polynomials of odd degree do not always result in zero integration error in the Lanczos framework. From the gap one may observe that the symmetry of Lanczos quadrature results in a higher convergence rate. Another work of ours discusses in detail the case of such a symmetry \cite{LZ25}.

Except for the explicit difference in Theorem \ref{thm:UCSthm4.2} and Theorem \ref{thm:CK21}, \cite{UCS17} stated that Theorem \ref{thm:UCSthm4.2} cannot be directly applied since the function $f$ in $\mathcal{I}$ is analytic on the \textit{physical domain} $\left[\lambda_{\min},\lambda_{\max}\right]$ rather than the \textit{reference domain} $\left[-1,1\right]$, thus an affine transformation is required. They claimed that after applying an affine transformation, there should be an additional factor in the bound, namely $(\lambda_{\max}-\lambda_{\min})/2$. \cite{CK21} considered it redundant, but left it unproved. We reexamine the discrepancy and explain the needlessness of the scalar mentioned above.

Let $h$ be an affine transformation that shifts the interval $\left[-1,1\right]$ to $\left[\lambda_{\min},\lambda_{\max}\right]$. Then $g$ is defined as $g = f \circ h$. Consider the integrals
$$\mathcal{I}=\int_{\lambda_{\min}}^{\lambda_{\max}}f(t)d\mu(t)$$ 
and
$$\mathcal{J}=\int_{-1}^{1}g(t)d\mu\left(h\left(t\right)\right),$$
where $\mu$ are defined in \eqref{eq:measure_f}. These integrals are in fact the finite sums of stepwise weighted increments, as
\begin{align}
    \mathcal{J}=\int_{-1}^{1}f(h(t))d\mu(h(t))&=\sum_{t_j:h(t_j)=\lambda_j}f(h(t_j))\left[\mu(h(t_j))-\mu(h(t_j)\text{-})\right]\notag\\
    &=\sum_{j=1}^nf(\lambda_j)\left[\mu(\lambda_j)-\mu(\lambda_j\text{-})\right]\notag\\
    &=\int_{\lambda_{\min}}^{\lambda_{\max}}f(t)d\mu(t)=\mathcal{I},\label{eq:wrongDerivation}
\end{align}
where $\mu(t\text{-})$ represents the left limit of the function $\mu$ on $t$. From this perspective, it is trivial that $\mathcal{I}=\mathcal{J}$. One may also use the delta function and derive
    \begin{align}
    \mathcal{J} &=\int_{-1}^{1}f(h(t))d\mu(h(t))
    =\int_{-1}^{1}f(h(t))\mu'(h(t))h'(t)dt \notag \\
    &=\int_{-1}^{1}f(h(t))\sum_{j=1}^n\mu_j^2\delta(h(t)-\lambda_j)\left(\frac{\lambda_{\max}-\lambda_{\min}}{2}\right)dt\notag \\        &=\left(\frac{\lambda_{\max}-\lambda_{\min}}{2}\right)\sum_{t_j:h(t_j)=\lambda_j}\mu_j^2f(h(t_j))
    =\left(\frac{\lambda_{\max}-\lambda_{\min}}{2}\right)\sum_{j=1}^n\mu_j^2f(\lambda_j)\notag \\
    &=\left(\frac{\lambda_{\max}-\lambda_{\min}}{2}\right)\mathcal{I}, \label{eq:correctDerivation}
    \end{align} 
which assumes the validity of the chain rule of measures
\begin{equation}
    \label{chain_rule}
    \begin{aligned}
        d\mu(h(t))&=\frac{d\mu}{dt}(h(t))\frac{dh}{dt}(t)dt\\
        &=\left(\frac{\lambda_{\max}-\lambda_{\min}}{2}\right)\sum_{j=1}^n\mu_j^2\delta(t-h^{-1}(\lambda_j))dt\\
        &=\left(\frac{\lambda_{\max}-\lambda_{\min}}{2}\right)d\mu(h(t)).
    \end{aligned}
\end{equation}
Here $d\mu(h(t))$ is the measure with the distribution function $\mu(h(t))$ and $dt$ is the Lebesgue measure. In fact, the chain rule of measures 
\begin{equation*}
    \frac{d\nu}{d\lambda}=\frac{d\nu}{d\mu}\frac{d\mu}{d\lambda}
\end{equation*}
is guaranteed by some assumptions on $\nu,\mu$ and $\lambda$. One such assumption is given in \cite{F99}.
\begin{proposition}\cite[Proposition 3.9]{F99}
    Suppose $\nu,\mu,\lambda$ are $\sigma$-finite measures on $(X,\mathcal{M})$ such that \begin{equation}\label{absolutely_continuous}
        \nu(E)=0\mbox{ for all }\mu(E)=0,\mbox{ and }\mu(E)=0\mbox{ for all }\lambda(E)=0,E\in\mathcal{M},
    \end{equation}
    then the Radon-Nikodym derivatives satisfy the chain rule
    \begin{equation*}
        \frac{d\nu}{d\lambda}=\frac{d\nu}{d\mu}\frac{d\mu}{d\lambda},\lambda-a.e.
    \end{equation*}
\end{proposition}
The measure in \eqref{chain_rule} does not satisfy assumption \eqref{absolutely_continuous} since $d\mu(h(t))$ assigns positive measure $\mu_j^2$ on single point $h^{-1}(\lambda_j)$, while the Lebesgue measure $dt$ assigns zero on any single point, which means the assumption of \eqref{chain_rule} is invalid. 

\section{Error Analysis for Asymmetric Stochastic Lanczos Quadrature}
\label{sec:ea}

After clarification in Section \ref{sec:basics}, we first review the error analysis with asymmetric Lanczos quadrature in log determinant estimation with random Rademacher vectors \cite{CK21}.
\begin{theorem}
\label{thm:CK21}
\cite[Theorem 5]{CK21}
    Let $A \in \mathbb{R}^{n \times n}$ be a real symmetric positive definite\footnote{For logarithm matrix function, the spectrum of $A$ should be on the positive half-axis. } matrix with eigenvalues in $[\lambda_{\min}, \lambda_{\max}]$, $\kappa = \lambda_{\max}/\lambda_{\min}$ be the condition number and $D_{\log(A)}$ be the diagonal matrix with $\log(A)$'s diagonal entries. Given $\epsilon > 0$, $\eta \in (0,1)$, if the Lanczos iteration $m$ and the Rademacher query number $N$ satisfy
    $$ m \ge \frac{\sqrt{\kappa + 1}}{4} \cdot \log\left(4\epsilon^{-1}n (\sqrt{\kappa + 1} + 1) \log(2 \kappa)\right), $$
    and
    $$ N\ge\frac{32}{\epsilon^2}\left(\Vert\log(A)-D_{\log(A)}\Vert_F^2 + \frac{\epsilon}{2}\Vert\log(A)-D_{\log(A)}\Vert_2\right) \log\left(\frac{2}{\eta}\right), $$
    then
    \begin{equation}
    \label{eq:PAEB}
        \mathbb{P}\Big\{ \left| \log \det A - \tr(\log(A))^{\dagger}_{\texttt{R}} \right| \le \epsilon \Big\} \ge 1-\eta. 
    \end{equation}
\end{theorem}
\begin{proof}
    See \cite[Section 4.2]{CK21}.
\end{proof}
% \begin{remark}
%     Note that in practice it might be difficult to know the stable rank $\rho_{\mathrm{logd}}$ a priori. \cite[Lemma 6]{CK21} also provided a substitute (but stronger) bound that avoids the use of that term.
% \end{remark}

\cite{CK21} demonstrated that their bound on $N$ benefits from the independence of dimension $n$ in the denominator. They also compared the lower bound of $m$ in Theorem \ref{thm:CK21} to the one proposed by \cite{UCS17} but with correction, showing a reduction of approximately $\sqrt{3}$. However, Theorem \ref{thm:CK21} is based on the equal split of error $\epsilon$ on both the Hutchinson trace estimator \eqref{err1} and the Lanczos quadrature \eqref{err2}. In order to guarantee \eqref{eq:PAEB}, the computational cost of Algorithm \ref{alg:SLQ} relates to the number of matrix-vector multiplications (MVMs), i.e., $\mathcal{O}(mN)$ if no re-orthogonalization is applied, and $\mathcal{O}(m^2N)$ vise versa. Thus, bounding both \eqref{err1} and \eqref{err2} by $\epsilon/2$ may not be the most computationally efficient approach. Instead, we can reallocate the error tolerance so that the number of MVMs is lessened. In brief, one should try to find a pair $(a, b)$ so that
\begin{equation}
    \label{opterr1}
        \mathbb{P}\left\{\left|\tr(\log(A))-\tr_{N}(\log(A))\right| \leq \frac{\epsilon}{a}\right\} \geq 1 - \eta,
\end{equation}
\begin{equation}
    \label{opterr2}
    \left|\tr_{N}(\log(A))-\tr(\log(A))^{\dagger}_{\texttt{R}}\right| \leq \frac{\epsilon}{b},
\end{equation}
with $\frac{1}{a} + \frac{1}{b} = 1, a > 1, b > 1$, so that
$$
\begin{aligned}
    1-\eta & \leq \mathbb{P}\Big\{ \left|\tr(\log(A)) - \tr_N(\log(A))\right|\leq \frac{\epsilon}{a}\Big\} \\
    &\leq \mathbb{P}\Big\{\left|\tr(\log(A)) - \tr_N(\log(A))\right| +\left|\tr_N(\log(A)) - \tr(\log(A))^{\dagger}_{\texttt{R}}\right| \leq \frac{\epsilon}{a} + \frac{\epsilon}{b}\Big\} \\
    & \leq \mathbb{P}\Big\{\left|\tr(\log(A)) - \tr(\log(A))^{\dagger}_{\texttt{R}}\right| \leq\epsilon\Big\}.
\end{aligned}
$$

In this paper, we apply the error allocation technique to reduce the computational cost required by Theorem \ref{thm:CK21} without re-orthogonalization. The key is to find $a^*$ that minimizes $mN$. In practice, one may follow the same procedure but find a minimizer that lessens $m^2N$ with orthogonalization.
\begin{corollary}
\label{thm:CK21_updated}
    Let $A \in \mathbb{R}^{n \times n}$ be a real symmetric positive definite matrix, $\kappa = \lambda_{\max}/\lambda_{\min}$ be the condition number, and $D_{\log(A)}$ be the diagonal matrix with $\log(A)$'s diagonal entries. Given $\epsilon>0$, $\eta \in (0,1)$, $C_1 = 2n\epsilon^{-1}(\sqrt{\kappa +1} + 1)\log(2\kappa)$, $C_2 = \Vert \log(A)-D_{\log(A)}\Vert_F^2$ and $C_3 = \epsilon\Vert\log(A) - D_{\log(A)}\Vert_2$, let
    $$ a^* = \arg\min_{a} \left\{  \log\left(\frac{C_1a}{a-1}\right)\left(C_2a^2+C_3a\right) \right\}. $$
    If
    \begin{equation}
        \label{eq:optlowerm}
        m \ge \frac{\sqrt{\kappa + 1}}{4}\log\left(\frac{C_1a^*}{a^* - 1}\right),
    \end{equation}
    and
    \begin{equation}
        \label{eq:optlowerN}
        N\ge \frac{8}{\epsilon^2}\left(C_2{a^*}^2 + C_3a^*\right)\log\left(\frac{2}{\eta}\right),
    \end{equation}
    then \eqref{eq:PAEB} holds.
\end{corollary}
\begin{proof}
    Let $\frac{1}{a}+\frac{1}{b}=1,$ and $a,b>1$. Based on \eqref{err2}, \eqref{opterr1} and \cite[Corollary 4]{CK21}, we have
    $$
    \begin{aligned}
        &\left|\tr_{N}(\log(A))-\tr(\log(A))^{\dagger}_{\texttt{R}} \right| \\
        & = \frac{n}{N}\sum_{i=1}^N\left|\mathcal{I}^{(i)}-\mathcal{I}_m^{(i)}\right|\\
        &\leq 2n\left(\sqrt{\kappa + 1} + 1\right)\log(2\kappa)\left(\frac{\sqrt{\kappa + 1} - 1}{\sqrt{\kappa + 1} + 1}\right)^{2m} \\
        &\leq \frac{\epsilon}{b}.
    \end{aligned}$$
    By simple derivations and $\log\left((\sqrt{\kappa + 1} + 1)/(\sqrt{\kappa + 1} -1)\right)\ge 2/(\sqrt{\kappa + 1})$, the lower bound of $m$ has the form of \eqref{eq:optlowerm}. On the other hand, based on \eqref{opterr2} and \cite[Corollary 1]{CK21} we require
    $$
    \begin{aligned}
        \mathbb{P}&\left\{\left|\tr(\log(A))-\tr_{N}(\log(A))\right| \geq \frac{\epsilon}{a}\right\} \\
        &\leq 2\exp\left(-\frac{N\epsilon^2}{8a^2\Vert \log(A) - D_{\log(A)}\Vert_F^2 + 8a\epsilon\Vert\log(A)-D_{\log(A)}\Vert_2}\right) \\
        &\leq \eta.
    \end{aligned}
    $$
    Based on the last inequality, we have \eqref{eq:optlowerN}. Ignoring the scalar that is not related to $a$ and $b$, the product of the lower bounds in \eqref{eq:optlowerm} and \eqref{eq:optlowerN} is 
    \begin{equation}
        \label{eq:optprod}
        \mathrm{Prod}(a) = \log\left(\frac{C_1a}{a-1}\right)\left(C_2a^2+C_3a\right),
    \end{equation}
    where $C_1 = 2n\epsilon^{-1}(\sqrt{\kappa +1} + 1)\log(2\kappa)$, $C_2 = \Vert \log(A)-D_{\log(A)}\Vert_F^2$ and $C_3 = \epsilon\Vert\log(A) - D_{\log(A)}\Vert_2$ are defined for readability. In order to minimize \eqref{eq:optprod}, we solve $\mathrm{Prod}'(a) = 0$, i.e., finding the zeros of
    $$ (2C_2 a+C_3)\log\left(\frac{C_1 a}{a-1}\right) = \frac{C_2a + C_3}{a - 1}. $$
    One may use numerical methods to find the root and substitute it back to \eqref{eq:optlowerm} and \eqref{eq:optlowerN} to minimize the theoretically needed MVMs.
\end{proof}

\cite{CK21} also established a probabilistic absolute error bound for the SLQ method using random Gaussian vectors.

\begin{theorem}
    \label{thm:CK21_Gauss}
\cite[Theorem 4]{CK21}
    Let $A \in \mathbb{R}^{n \times n}$ be a real symmetric positive definite matrix with eigenvalues in $[\lambda_{\min}, \lambda_{\max}]$, $\kappa = \lambda_{\max}/\lambda_{\min}$ be the condition number. Given $\epsilon > 0$, $\eta \in (0,1)$, if the Lanczos iteration $m$ and the Gauss query number $N$ satisfy
    $$ m \ge \frac{\sqrt{\kappa + 1}}{4} \cdot \log\left(4\epsilon^{-1}n^2 (\sqrt{\kappa + 1} + 1) \log(2 \kappa)\right), $$
    and
    $$ N\ge 16\epsilon^{-2}\left(\Vert\log(A)\Vert_2^2 + \frac{\epsilon}{2}\Vert\log(A)\Vert_2\right) \log(4/\eta), $$
    then
    \begin{equation}
    \label{eq:PAEB_Gauss}
        \mathbb{P}\Big\{ \left| \log \det A - \tr(\log(A))^{\dagger}_{\texttt{G}} \right| \le \epsilon \Big\} \ge 1-\eta. 
    \end{equation}
\end{theorem}
\begin{proof}
    See \cite[Section 4.1]{CK21}. Note that there is a typo in $N$'s lower bound in \cite[Theorem 4]{CK21} that $\epsilon$ should be half.
\end{proof}
Corollary \ref{thm:CK21_Gauss_updated} follows from applying the optimization technique to Theorem \ref{thm:CK21_Gauss}.
\begin{corollary}
    \label{thm:CK21_Gauss_updated}
    Let $A \in \mathbb{R}^{n \times n}$ be a real symmetric positive definite matrix, $\kappa = \lambda_{\max}/\lambda_{\min}$ be the condition number. Given $\epsilon>0$, $\eta \in (0,1)$, $C_1 = 2n^2\epsilon^{-1}(\sqrt{\kappa +1} + 1)\log(2\kappa)$, $C_2 = \Vert \log(A)\Vert_F^2$ and $C_3 = \epsilon\Vert\log(A)\Vert_2$, let
    $$ a^* = \arg\min_{a} \left\{  \log\left(\frac{C_1a}{a-1}\right)\left(C_2a^2+C_3a\right) \right\}. $$
    If
    \begin{equation}
        \label{eq:optlowerm_Gauss}
        m \ge \frac{\sqrt{\kappa + 1}}{4}\log\left(\frac{C_1a^*}{a^* - 1}\right),
    \end{equation}
    and
    \begin{equation}
        \label{eq:optlowerN_Gauss}
        N\ge \frac{4}{\epsilon^2}\left(C_2{a^*}^2 + C_3a^*\right)\log\left(\frac{4}{\eta}\right),
    \end{equation}
    then \eqref{eq:PAEB_Gauss} holds.
\end{corollary}
\begin{remark}
 It should be noted that the objective function to be minimized in Corollary \ref{thm:CK21_Gauss_updated} possesses the same functional form as that in Corollary \ref{thm:CK21_updated}. The optimization procedure remains identical, with the only distinction lying in the specific definitions of the constants $C_1, C_2,$ and $C_3$, which are adjusted to account for the properties of Gaussian random vectors.
\end{remark}

\section{Experiments}
\label{sec:experiments}
To evaluate the practical implications of our error allocation strategy, numerical experiments are conducted to compare the theoretical MVM requirements of Theorem \ref{thm:CK21}, Theorem \ref{thm:CK21_Gauss}, Corollary \ref{thm:CK21_updated} and Corollary \ref{thm:CK21_Gauss_updated} under the SLQ framework without re-orthogonalization. All simulations are implemented in MATLAB R2026a to verify the reduction in computational cost while guaranteeing \eqref{eq:PAEB} and \eqref{eq:PAEB_Gauss}.

Following the benchmarking framework established in \cite{CK21}, we evaluate our method on four representative matrices. The failure probability is fixed at $\eta = 0.1$, while the relative error tolerance $\epsilon$ is adaptively adjusted according to the rank of the tested matrices.

\begin{enumerate}[label=(\alph*)]
    \item A synthetic SPD \texttt{harmonic} matrix
    $$A = H \Lambda_j H^T \in \mathbb{R}^{5000\times 5000}$$ with different decaying rates of eigenvalues
    $$ \Lambda = \mathrm{diag}(1, 1/2, \ldots, 1/5000) $$
    and the Householder matrix 
    $$H = I - \frac{2}{n}(\mathbf{1} \cdot \mathbf{1}^T).$$
    We set $\epsilon^* = \texttt{linspace(0.00001,0.001,100)}$ and fix $\eta = 0.1$.
    \item A synthetic SPD \texttt{lowrank} matrix $A \in \mathbb{R}^{5000 \times 5000}$,
    $$ A = I + \sum_{i=1}^{40}\frac{10}{i^2}\boldsymbol{x}_i\boldsymbol{x}_i^T + \sum_{i=41}^{300} \frac{1}{i^2}\boldsymbol{x}_i \boldsymbol{x}_i^T, $$
    where $\boldsymbol{x}_i, i = 1, \ldots, 300$ are generated by $\texttt{sprandn(5000,1,0.025)}$ in MATLAB. This example comes from \cite{AAI17,S16}. We set $\epsilon^* = \texttt{linspace(0.001,0.1,100)}$ and fix $\eta = 0.1$.
    \item An SPD matrix from $\texttt{thermomec\_TC}$ in \cite{DH11}. We set $\epsilon = \texttt{linspace(0.00001,0.001,100)}$ and fix $\eta = 0.1$.
    
    \item An SPD matrix from $\texttt{nd3k}$ in \cite{DH11}. We set $\epsilon = \texttt{linspace(0.001,0.1,100)}$ and fix $\eta = 0.1$.
\end{enumerate}

A few implementation details regarding the \texttt{thermomec\_TC} matrix warrant mention. Since the matrix is too large for an explicit computation of $\log(A)$, we adopt the exact norms $\Vert\log(A)\Vert_F = 1.72\times10^3$ and $\Vert \log(A) - D_{\log(A)}\Vert_F = 122.8$ previously reported in \cite{CK21}. Furthermore, evaluating $\Vert \log(A) - D_{\log(A)} \Vert_2$ directly for the computation of $C_3$ in Corollary \ref{thm:CK21_updated} is computationally prohibitive. To circumvent this difficulty, we use the substitution $C_3 = 2\epsilon\Vert\log(A)\Vert_2$. Table \ref{tab:summary_mat} summarizes the properties of the tested matrices.
\begin{table}[htbp]
  \centering
  \caption{Summary of the four matrices}
  \label{tab:summary_mat}
  
  \begin{tabular}{@{} ccccc @{}}
    \toprule % 顶端粗线
    Name & \texttt{harmonic} & \texttt{lowrank} & \texttt{thermomec\_TC} & \texttt{nd3k} \\
    \midrule 
    Size & 5000 & 5000 & 102158 & 9000 \\
    stable rank of $A$ & 1.645 & 1.166 & $5.127\times10^3$ & $1.543\times10^3$ \\ 
    $\tr(\log(A))$  & $-3.759\times10^{4}$ & 80.318 & $-5.468\times10^5$ & $1.328\times10^4$ \\
    $\kappa(A)$  & $5\times10^3$ & $1.238\times 10^3$ & 878.220 & $1.567\times10^7$ \\
    $\Vert\log(A)\Vert_F$  & 536.260 & 15.445 & $1.72\times10^3$ & 275.433 \\
    $\Vert \log(A) \Vert_2$  & 8.517 & 7.1216 & 7.698 & 11.7184 \\
    $\Vert \log(A) - D_{\log(A)}\Vert_F$  & 2.816 &  15.217 & 122.8 & 228.6992 \\
    $\Vert \log(A) - D_{\log(A)}\Vert_2$  & 1.991 & 6.9327 & $\le 15.395$ & 13.3729 \\
    \bottomrule 
  \end{tabular}
  
\end{table}

% optimized a的选择，以表格呈现

Figure \ref{fig:test} illustrates the performance of the proposed error reallocation-based MVM bounds relative to existing benchmarks. It is observed that for matrices across a broad range of condition numbers, the bounds derived from Corollary \ref{thm:CK21_updated} and Corollary \ref{thm:CK21_Gauss_updated} consistently yield tighter estimates than those provided in \cite{CK21}.

Specifically, for the \texttt{harmonic} and \texttt{thermomec\_TC} test cases, the SLQ estimator utilizing Rademacher queries demonstrates superior computational efficiency compared to its Gaussian counterpart, a trend that becomes more pronounced as the relative error tolerance increases. Conversely, for the \texttt{lowrank} and \texttt{nd3k} datasets, the gap between the two estimators diminishes, with both proposed bounds exhibiting comparable performance. Moreover, our empirical results consistently yield a small optimal value of $a^*$ that is extremely close to $1$ (e.g., $a^* \approx 1.02$). This directly translates to a highly imbalanced, yet optimal, error budget: nearly all of the error tolerance $\epsilon$ should be assigned to the stochastic trace estimator, demanding a strictly bounded, highly accurate Lanczos quadrature step. 

\begin{figure}[htbp]
    \centering
    \subfloat[\texttt{harmonic}]{\label{fig:1.1}\includegraphics[width = 0.5\textwidth]{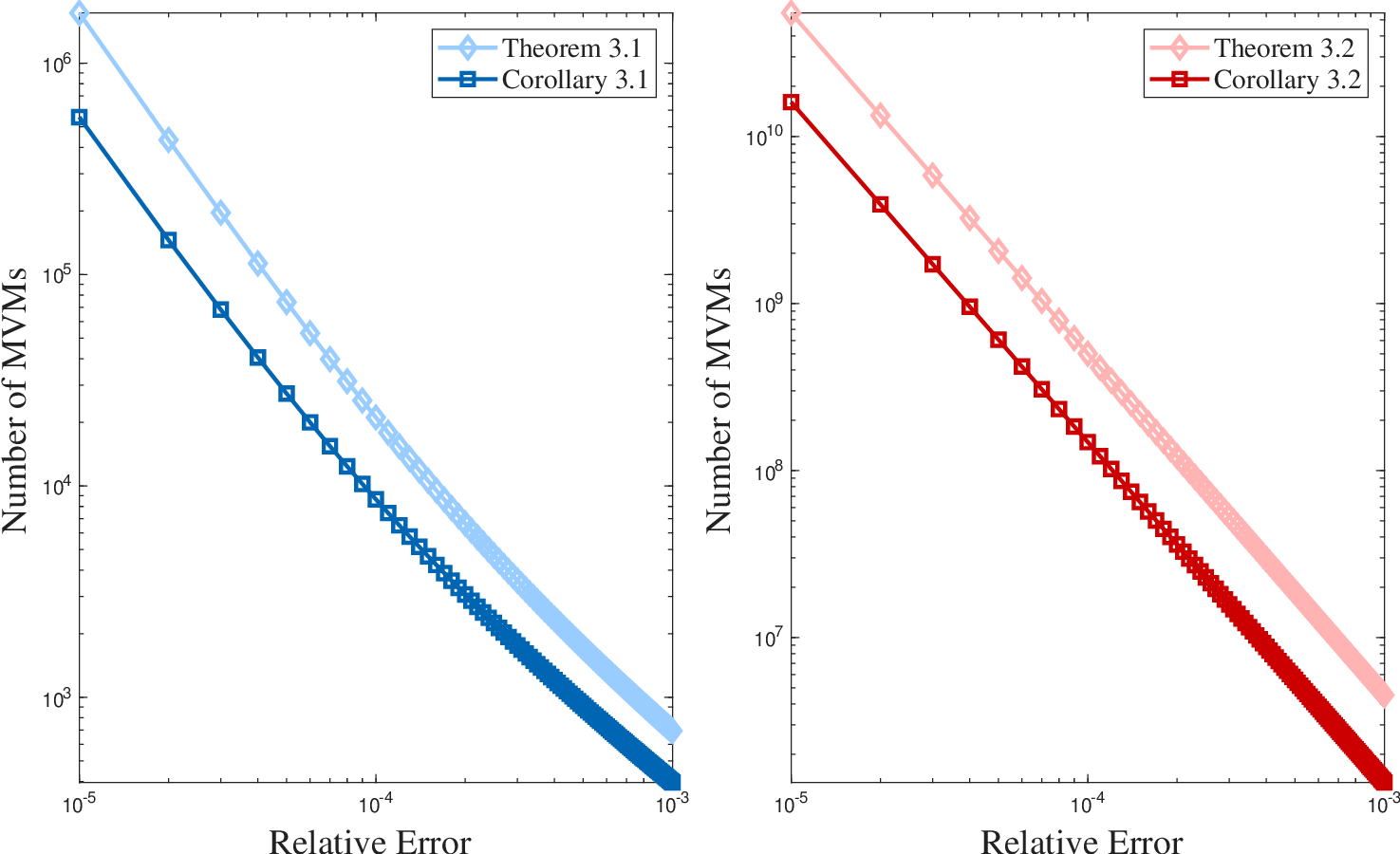}}\\
    \subfloat[\texttt{lowrank}]{\label{fig:1.2}\includegraphics[width = 0.5\textwidth]{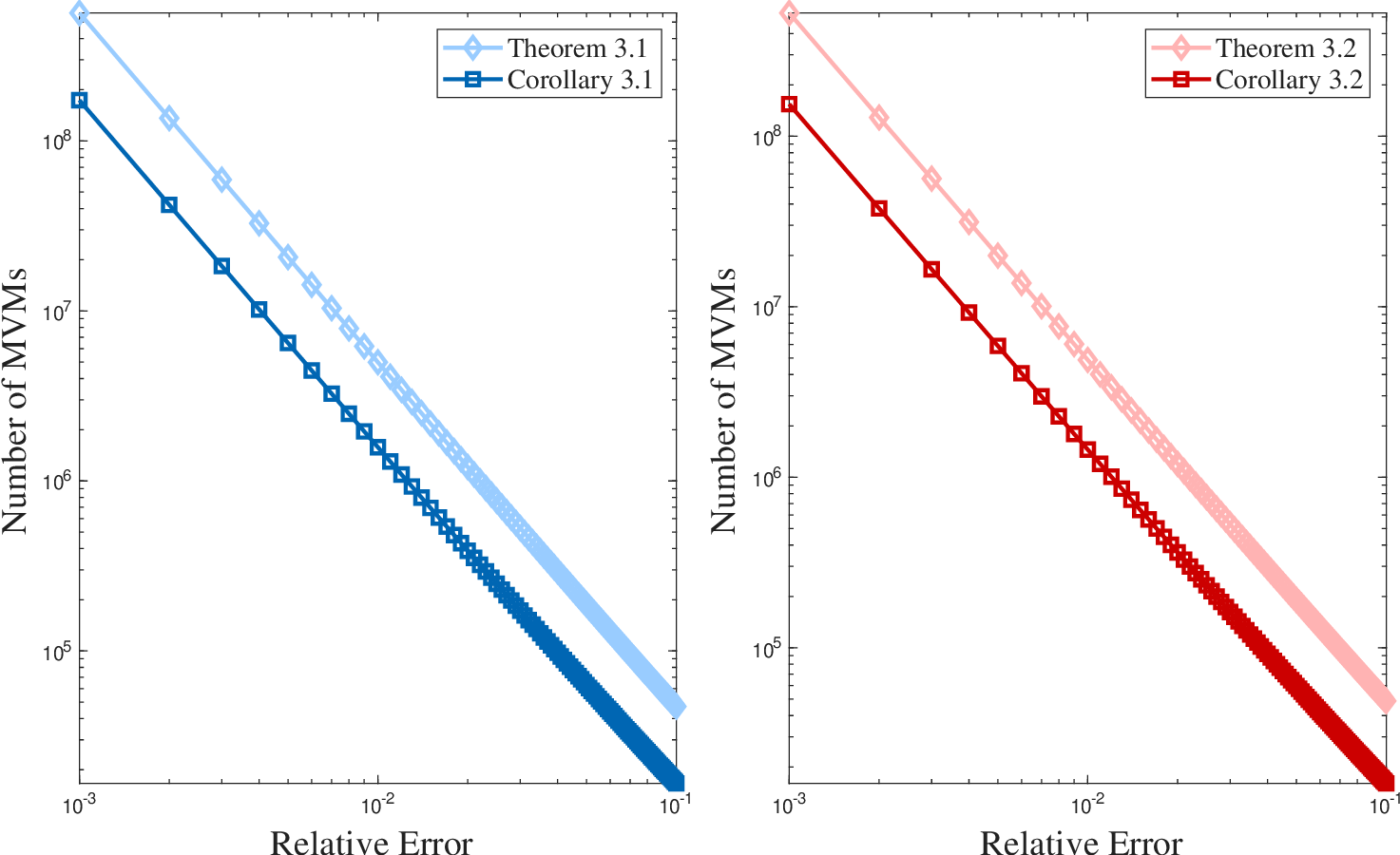}} \\
    \subfloat[\texttt{thermomec\_TC}]{\label{fig:2.1}\includegraphics[width = 0.5 \textwidth]{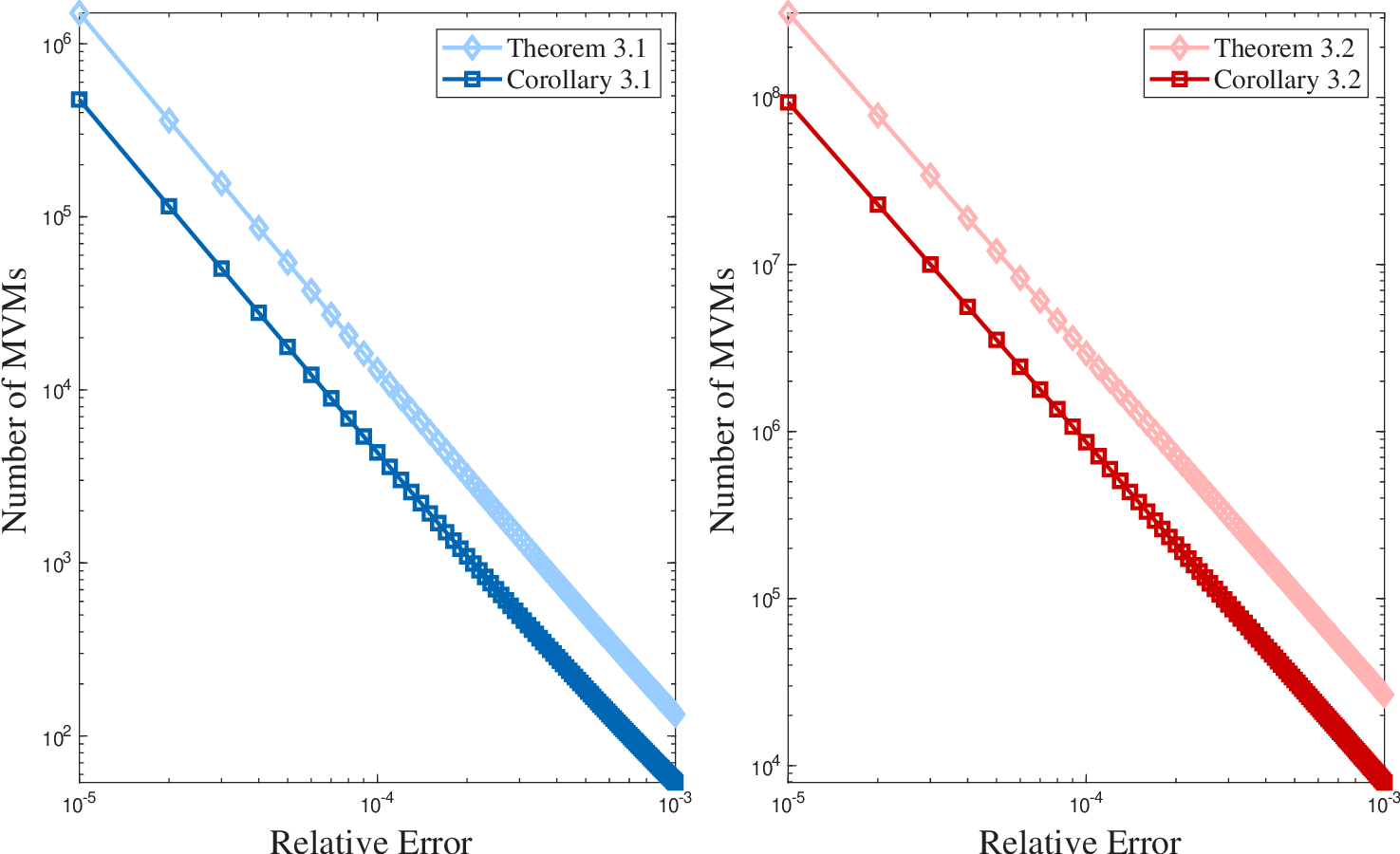}}\\
    \subfloat[\texttt{nd3k}]{\label{fig:2.2}\includegraphics[width = 0.5 \textwidth]{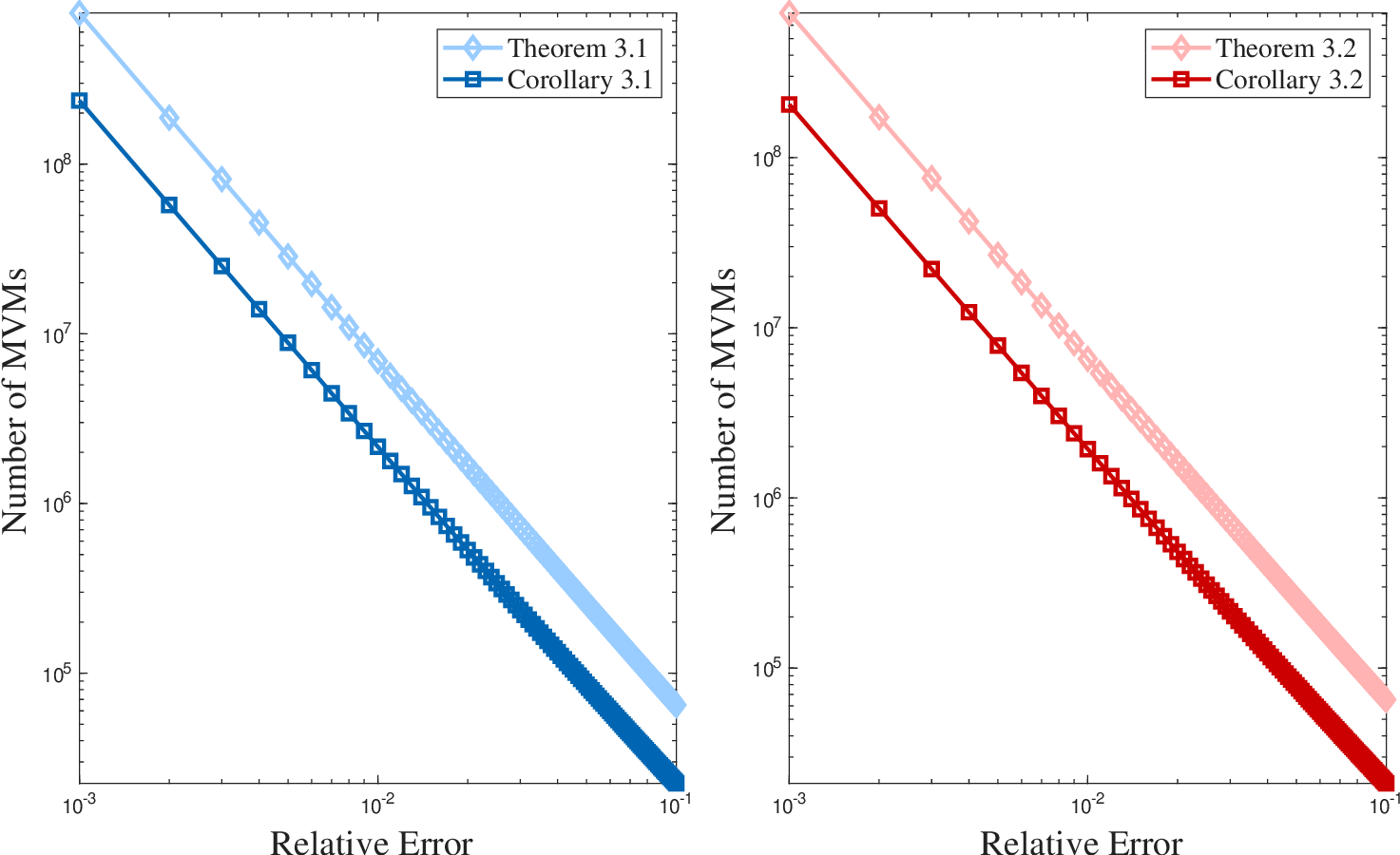}}
    \caption{Number of matrix-vector multiplications required by Theorem \ref{thm:CK21}, \ref{thm:CK21_Gauss} and Corollary \ref{thm:CK21_updated}, \ref{thm:CK21_Gauss_updated} versus relative error.}
    \label{fig:test}
\end{figure}

\section{Conclusion}
\label{sec:conclusion}
In this work, we have revisited the Stochastic Lanczos Quadrature (SLQ) framework, specifically the error bound of its approximation of trace estimation. A primary contribution of our study is the clarification of the discrepancy between the error bounds presented in \cite{UCS17} and \cite{CK21}. Our analysis reveals that this inconsistency stems from the underlying quadrature rules: the framework in \cite{UCS17} is inherently tied to symmetric Lanczos quadrature, whereas the derivations in \cite{CK21} pertain to asymmetric variants.

Furthermore, we introduced an optimized error reallocation technique designed to minimize the number of Matrix-Vector Multiplications (MVMs) required by both Rademacher and Gaussian-based estimators. Experimental results show that this reallocation strategy effectively tightens the theoretical upper bounds compared to standard allocations. This suggests that the efficiency of SLQ-based trace estimation depends not only on the choice of random vectors but also on the strategic distribution of error tolerances across the quadrature nodes. Furthermore, results suggest that a more efficient strategy entails allocating a stricter error tolerance to the Lanczos process while allowing a relatively higher budget for the Monte Carlo sampling.

We strongly recommend that future software implementations of the SLQ method default to this optimized parameter configuration $\{m, N\}$ to guarantee computational efficiency, particularly when evaluating dense matrix functions or operating within memory-constrained environments. Future works may also focus on the realization of the error reallocation technique in other frameworks such as Hutch++ \cite{M21}.

\bibliography{ref.bib}

\end{document}